\documentclass[10pt,a4paper]{article}

\usepackage[utf8]{inputenc}
\usepackage{amsmath}
\usepackage{amsfonts}
\usepackage{amssymb}

\usepackage{twoopt}

\usepackage{amsthm}

\usepackage{comment}
\usepackage{hyperref}

\newtheorem{theorem}{Theorem}[section]
\newtheorem{corollary}[theorem]{Corollary}
\newtheorem{lemma}[theorem]{Lemma}

\theoremstyle{definition}
\newtheorem{example}[theorem]{Example}
\newtheorem{question}[theorem]{Question}

\theoremstyle{remark}
\newtheorem*{remark}{Remark}

\newcommand{\defaultalphabet}{X}

\newcommand{\Gp}{\operatorname{Gp}}
\newcommand{\Mon}{\operatorname{Mon}}
\newcommand{\Inv}{\operatorname{Inv}}

\newcommand{\presentation}[3]{#1 \left\langle \, #2 \mid #3 \, \right\rangle}

\newcommandtwoopt{\gpPres}[2][R][\defaultalphabet]{\presentation{\Gp}{#2}{#1}}
\newcommandtwoopt{\monPres}[2][R][\defaultalphabet]{\presentation{\Mon}{#2}{#1}}
\newcommandtwoopt{\invPres}[2][R][\defaultalphabet]{\presentation{\Inv}{#2}{#1}}

\newcommand{\subGp}[1]{\Gp \left\langle #1 \right\rangle}
\newcommand{\subMon}[1]{\Mon \left\langle #1 \right\rangle}
\newcommand{\subInv}[1]{\Inv \left\langle #1 \right\rangle}

\newcommand{\pref}{\operatorname{pref}}
\newcommand{\red}{\operatorname{red}}

\newcommand{\freemon}[1][\defaultalphabet]{ \overline{#1}^{\ast}}

\newcommand{\FG}{\operatorname{FG}}

\title{
The word problem of finitely presented special inverse monoids via their groups of units

\footnotetext[0]{
$2020$ \textit{Mathematics Subject Classifaction}:
20F05,
20F10,
20M05,
20M18.
}


}

\author{
Jonathan Warne
\footnote{
Work undertaken while a PhD student at the University of East Anglia under the supervision of Robert Gray
}
}

\begin{document}

\maketitle

\begin{abstract}

	A special inverse monoid is one defined by a presentation where all the defining relations have the form $r = 1$.
By a result of Ivanov Margolis and Meakin the word problem for such an inverse monoid can often be reduced to the word problem in its maximal group image together with membership in a particular submonoid of that group, called the prefix monoid, being decidable.
We prove several results that give sufficient conditions for the prefix membership problem of a finitely presented group to be decidable.
These conditions are given in terms of the existence of particular factorisations of the relator words.
In particular we are able to find sufficient conditions for a special inverse monoid, its maximal group image and its group of units to have word problems that are algorithmically equivalent.
These results extend previous results for one-relator groups to arbitrary finitely presented groups.
We then apply these results to solve the word problem in various families of E-unitary special inverse monoids.
We also find some criteria for when amalgamations of E-unitary inverse monoids are themselves E-unitary.

\end{abstract}

\section{Introduction}

There have been several papers in the past 30 years looking at the conditions under which special inverse monoids have decidable word problem.
For instance Birget, Margolis and Meakin showed that having a single idempotent relator is sufficient \cite{BMM94}.
Hermiller, Lindblad and Meakin demonstrated that a single sparse relator was also sufficient \cite{HLM10}.
What might be called surface inverse monoids, those with presentations corresponding to those of surface groups, both of orientable \cite{IMM01} \cite{MMZ05} \cite{Mea05}
and non-orientable type \cite{DolGra21} have also been shown to have decidable word problems.
This focus on special inverse monoids is partially motivated by the fact that being able to decide word problem of special inverse monoids with a single reduced relator would imply the decidability of the word problem for one relation monoids \cite[Theorem $2.2$]{IMM01}.

	Each special inverse monoid has a natural association with two groups, its maximal group image and its group of units.
Therefore a reasonable question is what conditions on these groups are necessary and which sufficient for the special inverse monoid to have decidable word problem.

	Such an approach has proved very successful for finitely presented special monoids.
It was shown not only that the group of units could be expressed in a presentation with as few relations as the original monoid but that the monoid had decidable word problem if and only if its group of units has decidable word problem; this was shown in the one relator case by Adjan \cite{Adj66}, generalised to multiple relators by Makanin \cite{Mak66} and then later simplified by Zhang \cite{Zha92a} \cite{Zha92b}.

	For special inverse monoids however such a simple correspondence does not hold in general.
While being able to decide the word problem of the inverse monoid implies that one can decide the word problem of the group of units the converse is not always true, see for instance the undecidable one-relator inverse monoid in \cite{Gra20}.
It has also been shown by Gray, Silva and Sz\'akacs that there are inverse monoids with undecidable word problem whose maximal group images are hyperbolic and therefore possess decidable problem \cite{GSS22}.
Furthermore, consider the general form 
\begin{equation*}
M = \invPres[t_1 w_1 t_1^{-1} = 1, \ldots, t_k w_k t_k^{-1}, xx^{-1} = 1 = x^{-1}x \, (x \in X)][X, t_1, \ldots, t_k],
\end{equation*}
where $w_i \in \freemon$ for $1 \leq i \leq k$.
We may map this homomorphically to the bicyclic monoid $\invPres[t_i t_i^{-1} = 1][t_i]$ for any $1 \leq i \leq k$ so each $t_i$ is not a unit in $M$;
by combining this with \cite[Lemma $4.2$]{IMM01} it may be deduced that the group of units is $\gpPres[][X] \leq M$ and therefore has decidable word problem.
The inverse monoid $M$ has maximal group image $G = \gpPres[ w_1 = 1, \ldots, w_k = 1][X, t_1, \ldots, t_k]$.
This means we can choose $G$ to have undecidable word problem despite $U_M$ being a free group.
Thus we see that the decidability of the word problems of the special inverse monoid, its maximal group image and its group of units are in general independent.

	Therefore we may ask under what conditions the inverse monoid, its maximal group image and its group of units do have equivalent word problem.
A result of Ivanov, Margolis and Meakin \cite{IMM01} says that a special inverse monoid's maximal group image having decidable word problem and prefix membership problem is sufficient to decide the word problem for the E-unitary maximal image of the inverse monoid.
Using this Dolinka and Gray \cite{DolGra21} were able to find some conditions under which the word problem of the special inverse monoid is decidable in the one-relator case.
Further Gray and Ruskuc \cite{GraRus23} established a set of Makanin-style results which provide a finite presentation for the group of units of certain special inverse monoids.

	The primary aim of this paper is to extend the Dolinka and Gray \cite{DolGra21} results from the one-relator to the finitely presented case, that is to find conditions under which the word problem of a finitely presented special inverse monoid has decidable word problem if its maximal group image does.
We also synthesise this with some other results, including Gray and Ruskuc \cite{GraRus23}, to find conditions under which the word problems of a special inverse monoid, its maximal group image and its group of units are equivalent.
This approach is complicated at certain points by the ``loss'' of certain one-relator group results.
For instance one cannot verify the E-unitarity of examples solely by checking if the relators are cyclically reduced as may be done in the one-relator case due to \cite[Theorem $4.1$]{IMM01} nor can one use the Freiheitsatz to show that each factor is of infinite order in the disjoint alphabet case.

	The main results are chiefly relevant to E-unitary inverse monoids.
In order to produce examples to which the results may be applied we need to be able to find finitely presented E-unitary special inverse monoids with more than one relator.
To do this we develop the theory of Stephen \cite{Step98} regarding when amalgams preserve the E-unitary property.
We show that his general but opaque condition has as a consequence several specific naturalistic conditions whose fulfilment is easier to prove.
	
	In section $2$ we lay out some basic definitions as well as a number of key known theorems.
Section $3$ elaborates on a theorem of Stephen regarding which amalgamations of E-unitary inverse monoids are themselves E-unitarity.
Section $4$ is focused on the prefix monoid and the related property of conservative factorisation.
Sections $5$ and $6$ prove some decidability results for presentations with uniquely marked and alphabetically disjoint factorisations respectively.
We conclude with some remarks about the limits of these results and possible extensions.

\section{Preliminaries}

	\paragraph{Words}
	Let $X$ be an alphabet (a non-empty set of letters) and let $X^{\ast}$ denote the \textit{free monoid} which is the set of words written over this alphabet including the empty word, denoted $1$.
As we are concerned here with groups and inverse monoids we introduce $\overline{X} = X \cup X^{-1}$, where $X^{-1} = \lbrace x^{-1} \mid x \in X \rbrace$, with the obvious bijective correspondence between $X$ and its `copy'.
This may be extended from letters to words when we say that for $w \equiv x_1^{\varepsilon_1} x_2^{\varepsilon_2} \ldots x_k^{\varepsilon_k}$, where $\varepsilon_i \in \lbrace -1, 1 \rbrace$ and $x_i \in X$, the corresponding inverse is $w^{-1} \equiv x_k^{-\varepsilon_k} \ldots x_2^{-\varepsilon_2} x_1^{-\varepsilon_1}$.

	We use $w(x_1, x_2, \ldots, x_k)$ to indicate a word written over a certain set of letters and their inverses, i.e. that $w \in \freemon[ \lbrace x_1, x_2, \ldots, x_k \rbrace ]$.
In a similar manner for a set of words $u_1, u_2, \ldots, u_k$ we use $w(u_1, u_2, \ldots, u_k)$ to indicate that the word $w$ may be written over the pieces $u_i$ and their inverses in the free monoid.
Substitutions of pieces in such words work as might be expected, for instance if it is established that $w(x,y)$ indicates the word $xyx^{-1}$ then $w(u,v)$ would indicate $uvu^{-1}$. 
We also assume that for an individual word $w(x_1, \ldots, x_k)$ that there is no subset of the $x_j$ such that the word may be written over that subset, i.e. for each $1 \leq j \leq k$ there is at least one appearance of $x_j$ or $x_j^{-1}$.
If we say that there are a set of words $w_i(x_1, \ldots, x_k)$ for $i \in I$ then we only assume that there is no subset of the $x_j$ such that all the $w_i$ may be written over that subset, i.e. for each $1 \leq j \leq k$ there is at least one word $w_i$, where $i \in I$, in which $x_j$ or $x_j^{-1}$ appears.

	\paragraph{Congruences and Presentation}
	Let $\rho \subseteq \freemon \times \freemon$ be a set of pairs of words which form an equivalence relation.
If this equivalence relation is also closed under multiplication, $(u_1,v_1),(u_2,v_2) \in \rho \Rightarrow (u_1u_2,v_1v_2 \in \rho)$ then we call $\rho$ a \textit{congruence}.
In this case we may form the object $\freemon / \rho$ by equating any words $u,v \in \freemon$ such that $(u,v) \in \rho$.

	We may speak of a congruence $\rho$ generated by a set of pairs of words $(u_i, v_i) \subset \freemon \times \freemon$ for $i \in I$.
When we do this we mean that $\rho$ is the smallest congruence containing all the generating pairs.
For an example let $\rho$ be the congruence generated by $(xx^{-1},1)$ and $(x^{-1}x, 1)$ for all $x \in X$.
The object produced by the quotient of $\freemon$ by this congruence is the free group on $X$, denoted $\FG(X)$.
If we have some congruence $\rho^{\prime}$ generated by this $\rho$ and the set of pairs $(u_i,v_i)$, for $i \in I$, then we may express the object $\freemon / \rho^{\prime}$ as
\begin{equation*}
	G = \gpPres[u_i = v_i \, (i \in I)]
\end{equation*}
which we call a \textit{group presentation}.

	In a similar manner we may define presentations for \textit{inverse monoids}, which model partial bijections in the same sense that groups model full bijections.
Let $\rho$ be the congruence generated by $(uu^{-1}u,u)$ and $(uu^{-1}vv^{-1},vv^{-1}uu^{-1})$ for $u,v \in \freemon$, we call this the \textit{Wagner congruence} on $X$.
Similarly to groups if we let $\rho^{\prime}$ be the smallest congruence generated by $\rho$ and some set of pairs of words $(u_i, v_i)$ then we may express the object $\freemon / \rho^{\prime}$ as
\begin{equation*}
	M = \invPres[u_i = v_i \, (i \in I)]
\end{equation*}
which we call an \textit{inverse monoid presentation}.
If we have a inverse monoid of the form $\invPres[w_i = 1 \, (i \in I)]$ then we call this a \textit{special inverse monoid}.

	For an inverse monoid $M$, we say that $m \in M$ is a \textit{left units} if it is such that $m^{-1} m = 1$, a \textit{right unit} if $mm^{-1} = 1$ and a \textit{unit} if it is both a left unit and a right unit.
We will use $L_M$, $R_M$ and $U_M$ to denote to the sets of all left units, all right units and all units of $M$ respectively.

	We call a group or inverse monoid finitely presented if it may be presented using a finite set of generators and defining relations and we will be dealing exclusively with finite presentations throughout.

	\paragraph{E-unitary Inverse Monoids}

	For an inverse monoid $M$ there is the notion of a \textit{minimum group congruence} $\sigma$ which is the smallest congruence such that $G = M / \sigma$ is a group, we call this $G$ the \textit{maximal group image}.
A special inverse monoid $M = \invPres[w_i = 1 \, (i \in I)]$ has maximal group image $G = \gpPres[w_i = 1 \, (i \in I)]$.

	For a given inverse monoid $M$, we say that an element $e \in M$ is \textit{idempotent} if $e^2 = e$ and designate the set of all idempotents in $M$ by $E_M$.
An inverse monoid $M$ is said to be \textit{E-unitary} when an element of $M$ is $\sigma$ congruent to $1$ if and only if it is idempotent, or equivalently when $\sigma^{-1}(1) = E_M$.

	One of our most useful criterea for E-unitarity comes from Ivanov, Margolis and Meakin \cite[Theorem $4.1$]{IMM01}.
\begin{theorem}		\label{thm:cycRedIsEu}
	Let $M = \invPres[w=1]$ be an inverse monoid.
If $w \in \freemon$ is a cyclically reduced word then $M$ is E-unitary.
\end{theorem}
In the same paper it is also shown, by example, that all relations being cyclically reduced is insufficient, even in the two relator case, to guarantee the inverse monoid being E-unitary.

	Let $a \sim b$ if and only if $ab^{-1}$ and $b^{-1}a$ are both idempotents, we call this the \textit{compatibility relation}.
This relation is only transitive in E-unitary inverse monoids, moreover an inverse monoid is E-unitary if and only $\sigma = \sim$, that is if its minimum group congruence is the same as the compatibility relation.
This means that for E-unitary $M$ with maximal group image $G$, $u = v$ in $G$ if and only if $u \sim v$ in $M$.

	The above paragraph and the rest of this subsection draws on a work of Lawson \cite{Law16} (available online) which itself draws its material primarily from his book \cite{Law98}.
The following is a standard result, the proof of which may be found in \cite[Lemma $2.10$]{Law16}.
\begin{lemma} 	\label{lem:equivPOconds}
	Let $M$ be an inverse monoid and let $x, y \in M$.
Then the following are equivalent:
\begin{itemize}
	\item That $x = ye$ for some $e \in E_M$.
	\item That $x = fy$ for some $f \in E_M$.
	\item That $x = xx^{-1}y$.
	\item That $x = yx^{-1}x$.
\end{itemize}
\end{lemma}
By saying that $x \leq y$ if the elements satisfy these (equivalent) conditions, we may form a partial order on $M$.
For a partial order we can speak of the concept of the \textit{meet} of two elements, denoted $x \wedge y$, which (when it is definable) is the greatest lower bound of the two elements.
That is if $z = x \wedge y$ then $z \leq x, y$ and for all $z^{\prime} \leq x, y$ we also have that $z^{\prime} \leq z$.
It is important to note that $x \wedge y$ is not necessarily defined for all pairs of elements of an inverse monoid.
In fact the existence of a meet is tied to the idea of compatibility.

	The following result essentially follows from Lawson \cite[Lemma $2.16$]{Law16}, for completeness we provide a proof.

\begin{lemma}	\label{lem:meetIFFcompatible}
	Let $m,n$ be elements of an E-unitary inverse monoid $M$.
Then $m \wedge n$ exists if and only if $m \sim n$.
Moreover if $m \wedge n$ exists then it is equal to all of the following
$mm^{-1}n$, $nn^{-1}m$, $mn^{-1}n$ and $nm^{-1}m$.
\end{lemma}

\begin{proof}
	Suppose that $m \wedge n$ exists.
Put $z = m \wedge n$.
By definition $z \leq m$, which implies that $z = zz^{-1}m$.
From this we get that $zm^{-1} = zz^{-1}mm^{-1}$ and that $z^{-1}m = m^{-1}zz^{-1}m$, inspection shows that both of these are idempotents, and so $z \sim m$.
Dually $z \sim n$ and as the compatibility relation is transitive in E-unitary monoids we have that $m \sim n$.

	Conversely suppose that $m \sim n$, then $m^{-1}n$ is an idempotent.
Set $z = mm^{-1}n$, then $z \leq n$ and $z \leq m$, since $mm^{-1}$ and $m^{-1}n$ are idempotent.
Let $x \in M$ be such that $x \leq m$ and $x \leq n$.
Then $x = xx^{-1}x \leq mm^{-1}n = z$ and so $z$ is the greatest lower bound of $m$ and $n$,
Therefore $m \wedge n$ exists and is equal to $mm^{-1}n$.
The other forms may be found by similar arguments.
\end{proof}

	\paragraph{Decidability}
	We say that a class of true/false questions are \textit{decidable} if there is an algorithm which can determine the answer to any of them in finite time.
For instance, given a group presentation $G = \gpPres[w_i = 1]$ we might ask whether it was decidable if some $w \in \freemon$ was equal to the empty word $1$.
This is called the \textit{word problem} for $G$ and Magnus \cite{Mag32} showed that it is decidable for groups with a single relation.
However, it has also been shown, independently, by Novikov \cite{Nov55} and Boone \cite{Boo57} that the word problem is undecidable in general for finitely presented groups.

	A second class of problems we are interested in are \textit{membership problems}.
If we have a group $G = \gpPres$ and a set $Q \subset G$ then we can decide membership in $Q$ within $G$ if there is an algorithm that takes a word $w \in \freemon$ and in finite time outputs yes if $w$ represents an element of $Q$ within $G$ and no if it does not.
Importantly this does not require us to know which specific element of $Q$ it may represent.

	For words $w, p \in \freemon$ we have the natural notion of $p$ being a prefix if $w \equiv pq$ for $q \in \freemon$, we use $\pref(w)$ to denote the set of prefixes of $w$.
For a group presentation $G = \gpPres[w_i = 1 \, (i \in I)]$ we call
\begin{equation*}
	P = \subMon{\pref(w_i) \, (i \in I)} \leq G
\end{equation*}
the \textit{prefix monoid} of $G$.
It is important to stress here that we do not assume, either in this instance or throughout the rest of the paper, that the $w_i$ are reduced or cyclically reduced just because they are group relators.
We may then ask, for a given group $G$ is membership of $P$ within $G$ decidable, that is for $w \in \freemon$ can we decide if $w$ represents an element of the prefix monoid in $G$.
We call this the \textit{prefix membership problem} and it is dependent on our choice of presentation for $G$
(see, for instance, $\gpPres[aab = 1][a,b]$ and $\gpPres[aba = 1][a,b]$; while both are isomorphic to $\mathbb{Z}$, under the mapping $a \mapsto 1$, $b \mapsto -2$, their prefix monoids would be equal to $\mathbb{N}$ and $\mathbb{Z}$ respectively).

	Another natural class of questions we might ask is: given an inverse monoid $M = \invPres[w_i = 1 \, (i \in I)]$, is it decidable whether $u = v$ in $M$ for $u,v \in \freemon$?
This is the word problem for special inverse monoids (note that while we could formulate the word problem for groups in the same manner it would reduce to way we have phrased it above due to the invertibility of group elements).
This is a question that has been connected to the two previously mentioned problems by another result of Ivanov, Margolis and Meakin \cite[Theorem $3.3$]{IMM01}.
If we let $M_{EU}$ be the maximal E-unitary image, the largest homomorphic image of $M$ which is E-unitary, analogous in the obvious way to the maximal group image, then they show the following.
\begin{theorem}	\label{thm:PMP+GWP=IWP}
	Let $M = \invPres[w_i = 1 \, (i \in I)]$ be an inverse monoid and let $G = \gpPres[w_i = 1 \, (i \in I)]$ be its maximal group image.
If the word and prefix membership problems for $G$ are decidable then the word problem for $M_{EU}$ is decidable, in particular $M$ has decidable word problem if it is E-unitary.
\end{theorem}
This result motivates our work here to determine when the prefix membership problem is decidable for finitely presented special group presentations.

	\paragraph{Amalgamations}
	Suppose we have two groups $G_1 = \gpPres[r_i=1 \, (i \in I_1)][X_1]$ and $G_2 = \gpPres[s_i=1 \, (i \in I_2)][X_2]$, where $X_1 \cap X_2 = \emptyset$.
The \textit{free product} of $G_1$ and $G_2$ is denoted by $G_1 \ast G_2$ and has the following presentation
\begin{equation*}
	 \gpPres[r_{i_1} = 1 \, (i_1 \in I_1), s_{i_2} = 1 \, (i_2 \in I_2)][X_1, X_2].
\end{equation*}
The following is a standard result, see for example \cite[Corrollary IV.1.3]{LynSch01}

\begin{lemma}	\label{lem:freeProdDec}
	Let $G_1$ and $G_2$ be groups with decidable word problem.
Their free product $G_1 \ast G_2$ has decidable word problem.
\end{lemma}

	We may also speak of a \textit{free product amalgamated over $K$} which we denote by $G_1 \ast_K G_2$.
As we will be amalgamating over finitely generated groups we will also use $G_1 \ast_{u_i = v_i} G_2$ to denote two groups amalgamated over the isomorphic subgroups $\subGp{u_i \, (i \in I)} \leq G_1$ and  $\subGp{v_i \, (i \in I)} \leq G_2$, for some finite index set $I$.
The group $G_1 \ast_K G_2$ has the following finite presentation 
\begin{equation*}
	\gpPres[r_{i_1} = 1 \, (i_1 \in I_1), s_{i_2} = 1 \, (i_2 \in I_2), u_i = v_i \, (1 \leq i \leq n)][X_1, X_2].
\end{equation*}

	We cite the following is a standard result (a proof of which may be found in Lipshutz \cite[Lemma 2]{Lip64}).

\begin{lemma}	\label{lem:amalProdDec}
	Let 
\begin{equation*}
	G_1 = \gpPres[r_{i_1} = 1 \, (i_1 \in I_1)][X_1]
	\text{ and } 
	G_2 = \gpPres[s_{i_2} = 1 \, (i_2 \in I_2)][X_2]
\end{equation*}
be finitely presented groups with decidable word problem.
Further let $u_1, \ldots, u_k \in \freemon[X_1]$ and $v_1, \ldots, v_k \in \freemon[X_2]$ be sets of words such that $\subGp{u_1, \ldots, u_k} \leq G_1$ and $\subGp{v_1, \ldots, v_k} \leq G_2$ are isomorphic under the mapping $\phi: u_i \mapsto v_i$ and have decidable membership within their respective groups.
Then $G_1 \ast_{u_i = v_i} G_2$ has decidable word problem.
\end{lemma}

	We use $\subInv{W} \leq M$ to represent the \textit{inverse submonoid} generated by $W \in \freemon$ within $M = \invPres$.
This inverse submonoid, like a subgroup but unlike a non-inverse submonoid, satisfies $\subInv{W} = \subInv{W \cup W^{-1}} \leq M$ (this is because groups and inverse monoids, unlike monoids, have a native inverse).
Thus we can talk of constructing free products of inverse monoids both with or without amalgamation by requiring any amalgamation to be over isomorphic inverse submonoids in the same manner that amalgamated group products require isomorphic subgroups.

\section{E-unitary Amalgamations of Inverse Monoids}

	In his paper Stephen \cite{Step98} introduced a sufficient condition for the amalgamated product of E-unitary inverse semigroups to be E-unitary itself, though he did not provide any examples which satisfied the condition.
In order to understand this condition we first need to define the following concept.

	We say that a subinverse semigroup $U \subseteq S$ is \textit{upwardly directed} into $S$ if for all $s \in S$ and $u \in U$ then $s \wedge u$ being defined implies that there exists $v \in U \cup \lbrace 1 \rbrace$ such that $s,u \leq v$.

\begin{theorem} \label{thm:StephensAmalgam}
	Let $S_1$ and $S_2$ be E-unitary inverse semigroups and let $U$ be an inverse subsemigroup of $S_1$ and $S_2$.
If $U$ is upwardly directed into $S_1$ and $S_2$ then $S_1 \ast_U S_2$ is E-unitary.
\end{theorem}

	In this paper we will be restricting our use of this result to inverse monoids rather than the broader class of subsemigroups.
We may do so because the result is not specific to presentation and any finitely presented inverse monoid may be finitely presented as an inverse semigroup.
This is the case because an inverse monoid with the presentation $\invPres$ has the following inverse semigroup presentation
\begin{equation*}
	\operatorname{InvSem} \langle X, e \mid R^\prime, e^2 = e, xe = x, x = ex \, (x \in X) \rangle
\end{equation*}
where $R^\prime$ is the result of replacing any relations of the form $r = 1$ with $r = e$.

We will now demonstrate that certain classes of inverse submonoids are upwardly directed into their inverse monoids.

\begin{lemma}	\label{lem:upDirConditions}
	Let $M$ be an E-unitary inverse monoid and $N$ be an inverse submonoid of $M$.
If any of the following hold:
\begin{enumerate}
	\item If $N \subseteq U_M$;
	\item If $N \subseteq E_M$;
	\item If $N = \subInv{r} \leq M$ where $r \in R_M$;
	\item If $N = \subInv{l} \leq M$ where $l \in L_M$.
\end{enumerate}
then $N$ is upwardly directed into $M$.
\end{lemma}

\begin{proof}

	Suppose that $N \subseteq U_M$.
Let $m \in M$ and $n \in N$.
By Lemma \ref{lem:meetIFFcompatible} we know that if $m \wedge n$ is defined then it is equal to both $nn^{-1}m$ and $mm^{-1}n$.
As $n$ is a unit the former is equal to $m$ and so we may say that $m = mm^{-1}n$.
Thus $m \leq n$ as $mm^{-1}$ is idempotent.
Therefore as $n,m \leq n \in N$ we have shown that $N$ is upwardly directed into $M$.

	Suppose that $N \subseteq E_M$.
Let $m \in M$ and $n \in N$ be such that $m \wedge n$ is defined.
As $n$ is an idempotent it may easily be deduced that $1 \sim n$.
Moreover as the meet of $m$ and $n$ is defined this means by Lemma \ref{lem:meetIFFcompatible} we have that they are compatible.
Thus we may write $m \sim n$.
As $M$ is E-unitary, compatibility is transitive and so $1 \sim m$ which implies that $m \in E_M$.
Therefore $m,n \leq 1 \in N$ and so $N$ is upwardly directed into $M$.

	Suppose that $N = \subInv{r} \leq M$ where $r \in R_M$.
Let $n \in N$ and $m \in M$ be such that $m \wedge n$ is defined.
As $n \in \lbrace r, r^{-1} \rbrace^{\ast}$ and $r r^{-1} =1$, $n$ must be of the form $r^{-k_1}r^{k_2}$ for some $k_i \geq 0$.
Moreover as $r^{-i}r^i$ is idempotent for all $i$, we may write $n = e r^k$ for some $k$ and idempotent $e \in N$.
It may be then be seen that $n \sim r^k$.
By Lemma \ref{lem:meetIFFcompatible} the existence of $m \wedge n$ means that $m \sim n$.
Thus $m \sim r^k$ and may be written $f r^k$ for some idempotent $f \in N$ (as $r^k \in R_M$ is maximal in the partial order by $R_M \cap E_M = \lbrace 1 \rbrace$).
Therefore $m,n \leq r^k \in N$ and so $N$ is upwardly directed into $M$.

	The proof for $N = \subInv{l} \leq M$ for $l \in L_M$ is dual to the above.

\end{proof}

	Using the first of these cases we will now produce a theorem which will prove useful later in the construction of examples.

\begin{theorem}		\label{thm:EuAmalByUM}
	Let $M = M_1 \ast_N M_2$ where $M_1$ and $M_2$ are E-unitary inverse monoids and $N$ is an inverse submonoid of both $M_1$ and $M_2$.
If $N \subseteq U_{M_1} \cap U_{M_2}$ then $M$ is E-unitary.
Further if $M_1$ and $M_2$ have finite special presentations and $N$ is finitely generated then there is a finite special presentation of $M$.
\end{theorem}

\begin{proof}

	That $M$ is E-unitary follows as a natural consequence of Theorem \ref{thm:StephensAmalgam} and Lemma \ref{lem:upDirConditions}.

	Suppose that $M_1$ and $M_2$ are finitely presented and $N$ is finitely generated.
Let $M_1 = \invPres[r_i = 1 (i \in I)][X_1]$ and $M_2 = \invPres[s_j = 1 (j \in J)]$.
Further let $N$ be generated by $u_1, \ldots u_k$ in $M_1$ and correspondingly by $v_1, \ldots, v_k$ in $M_2$.
Then $M$ may be written
\begin{equation*}
	\invPres[r_i = 1 (i \in I), s_j = 1 (j \in J), u_1 = v_1, \ldots, u_k = v_k][X_1, X_2]
\end{equation*}
As $N \subseteq U_{M_2}$, the $v_i$ are invertible and so $u_i = v_i$ if and only if $u_i v_i^{-1} = 1$.
Consequently we have
\begin{equation*}
	M = \invPres[r_i = 1 (i \in I), s_j = 1 (j \in J), u_1 v_1^{-1} = 1, \ldots, u_k v_k^{-1} = 1][X_1, X_2]
\end{equation*}
which is a finite special presentation.

\end{proof}

	Let $M = \invPres[r_i(u_1, \ldots, u_k) = 1 \, (i \in I)]$ be an inverse monoid, we say that the factorisation of the $r_i$ into $u_j$ is \textit{unital} when each $u_j$ is a unit in $M$ (we take our lead here from the language of Dolinka and Gray \cite{DolGra21}).
We call this a factorisation into \textit{minimal invertible pieces} if no factor has a non-trivial prefix which is invertible in the monoid.
There are a number of methods for determining a unital factorisations, the simplest of which is iteratively finding words which are both prefixes and suffixes of known units starting with the relators.
However no method is currently known is guaranteed to produce a decomposition into minimal invertible pieces, even in the one-relator case, for more on this see \cite{Lal74},\cite{GraRus23} and \cite{CFNB22}.
We can use this notion of unital factorisations to find applications for Theorem \ref{thm:EuAmalByUM}.

\begin{example}		\label{exm:EuForUMLexample}
	Consider the inverse monoid
\begin{align*}
	M = \Inv \langle 
	x_1,y_1,z_1,x_2,y_2,z_2
	\mid
	(z_1)(x_1^2y_1)^2(z_1)=1,
	(z_2)(x_2^2&y_2)^2(z_2)=1, \\
	& (z_1)(z_2)^{-1}=1
	\rangle
\end{align*}
for which the brackets mark a unital factorisation.
This is the product of two inverse monoids $M_i = \invPres[(z_i)(x_i^2y_i)^2(z_i) = 1][x_i,y_i,z_i]$ (for $i=1,2$) amalgamated over $z_1 = z_2$.
As $M_1$ and $M_2$ are copies of each other the inverse submonoids generated by $z_1$ and $z_2$ are clearly isomorphic (and finitely generated).
Moreover as each $M_i$ is defined by a single cyclically reduced relator we know that they are both E-unitary by Theorem \ref{thm:cycRedIsEu}.
Therefore by Theorem \ref{thm:EuAmalByUM} $M$ is E-unitary.
\end{example}

\begin{example}		\label{exm:EuForDAexample}
	For $i=1,2$, let
\begin{equation*}
	M_i = \invPres[(a_i^2 b_i^3) (c_i^5)^2 (a_i^2 b_i^3) (c_i^5) (a_i^2 b_i^3) = 1][a_i, b_i, c_i]
\end{equation*}
be two inverse monoids.
Then we can find that $u_i \equiv a_i^2 b_i^3$ and $v_i \equiv c_i^5$ are unital pieces.
Therefore we can show that
\begin{align*}
	M = \Inv \langle a_1, b_1, c_1, a_2, b_2, c_2 \mid
	(a_1^2 b_1^3) (c_1^5)^2 (a_1^2 b_1^3) (c_1^5) (a_1^2 b_1^3) & = 1, \\
	(a_2^2 b_2^3) (c_2^5)^2 (a_2^2 b_2^3) (c_2^5) & (a_2^2 b_2^3) = 1, \\
	& (a_1^2 b_1^3) (a_2^2 b_2^3)^{-1} =1 \rangle
\end{align*}
is E-unitary in a similar manner to above.
\end{example}

\section{Properties of the Prefix Monoid}

	There is a strong link between the right units of a special inverse monoid and the prefix monoid of the maximal group image.
This link is even stronger when the monoid in question is E-unitary.

\begin{lemma}	\label{lem:sigmaRMisP}

	Let $M = \invPres[s_i = 1 \, (i \in I)]$ and $G = \gpPres[s_i = 1 \, (i \in I)]$ be an inverse monoid and its corresponding maximal group image.
Further let $\sigma$ be the map from $M$ to $G$ induced by identity on $\overline{X}$.

	If $P$ is the prefix monoid of $G$ then $P = \sigma(R_M)$ and furthermore if  $M$ is E-unitary then $\sigma^{-1}(P) = R_M \cdot E_M = R_M \cdot E_M$.

\end{lemma}

\begin{proof}

	Let $\tau$ and $\pi$ be the natural maps from $\freemon$ to $M$ and $G$ respectively.

	Let $p \in P$, by definition of the prefix monoid this means that we may find some $w_p \equiv w_1 w_2 \ldots w_k$ such that $\pi(w_p) = p$ and each $w_j$ is a prefix of one of the $s_i$.
As prefixes of the relators are right invertible in a special monoid we also know that each $\tau(w_j)$ is right invertible in $M$.
As the product of right invertible is right invertible we have that $\tau(w_p) \in R_M$.
Therefore $p = \sigma( \tau(w_p) ) \in \sigma( R_M )$ and so $P \subseteq \sigma( R_M )$.

	Let $m \in R_M$.
By a geometric argument of Ivanov, Margolis and Meakin \cite[Paragraph 2 of Lemma 4.2]{IMM01},  every right unit is a product of prefixes.
Therefore we may find some $w_m \equiv w_1 w_2 \ldots w_k$ such that $\tau(w_m) = m$ and each $w_j$ is a prefix of one of the $s_i$.
This means that
\begin{equation*}
	\sigma (m) = \sigma( \tau (w_m) )
	= \sigma( \tau (w_1 w_2 \ldots w_k) )
	= \sigma( \tau (w_1) ) \sigma( \tau (w_2) ) \ldots \sigma( \tau (w_k) ) \in P
\end{equation*}
as each $\sigma( \tau(w_i))$ will be a generator of $P$.
Therefore $\sigma(R_M) \subseteq P$ and so $P = \sigma(R_M)$.

	We now suppose that $M$ is E-unitary.

	Let $m \in \sigma^{-1}(P)$.
By the prior result this means that $m \in \sigma^{-1}( \sigma(R_M) )$.
Therefore there is some $r \in R_M$ such that $\sigma(m) = \sigma(r)$.
It follows that $\sigma(r^{-1} m) = \sigma(1)$ and, as $M$ is E-unitary, that $r^{-1} m \in E_M$.
Therefore, as $r$ is right invertible, $m = r \cdot r^{-1} m \in R_M \cdot E_M$ and so $\sigma^{-1}(P) \subseteq R_M \cdot E_M$.

	Let $m \in R_M \cdot E_M$, then there exist $e \in E_M$ and $r \in R_M$ such that $m = er$.
By the E-unitarity of $M$ we may see that $\sigma(m) = \sigma(re) = \sigma(r)$ and so $m \in \sigma^{-1} ( \sigma( R_M ) )$ which means by the above that $m \in \sigma^{-1}( P )$.
Therefore $R_M \cdot E_M \subseteq \sigma^{-1}(P)$ and so $\sigma^{-1}(P) = R_M \cdot E_M$.

	We will now prove that for any subset $N \subseteq M$ we have $N \cdot E_M = E_M \cdot N$.
Let $m \in N \cdot E_M$, this means there are some $n \in N$ and $e \in E_M$ such that $m = ne$.
It follows that $m = nn^{-1}ne$, further as $n^{-1}n$ is an idempotent it commutes with $e$ giving us $m = nen^{-1}n$.
The element $nen^{-1}$ is also idempotent and so $m = (nen^{-1})n \in E_M \cdot N$.
This means that $N \cdot E_M \subseteq E_M \cdot N$.
That $E_M \cdot N \subseteq N \cdot E_M$ follows by a dual argument.

	Hence $R_M \cdot E_M = E_M \cdot R_M$.

\end{proof}

\begin{remark}
	If the word problem for $M$ is decidable then deciding membership in both $E_M$ and $R_M$ is an easy consequence.
However it is not clear whether deciding membership in $E_M \cdot R_M$ is also a consequence.
This is notable as by the above result saying whether it was would be equivalent to answering the open question \cite{HLM10} of whether being able to decide the prefix membership problem is a necessary condition to deciding the word problem of $M$ when $M$ is E-unitary.
\end{remark}

	Let $G = \gpPres[r_i(u_1, \ldots, u_k) = 1 \, (i \in I)]$ be a group.
Further let $V \subseteq \overline{\lbrace u_1, \ldots, u_k \rbrace}$ be minimal such that $r_i \in V^\ast$ for every $i \in I$.
The factorisation into $u_1, \ldots, u_k$ is said to be \textit{conservative} in $G$ if
\begin{equation*}
	\subMon{\bigcup_{i \in I}  \pref(r_i) } 
	= \subMon{\bigcup_{v \in V} \pref(v) }
	\leq G
\end{equation*}
that is, if the submonoid generated by the prefixes of the factors is equal to the prefix monoid, the monoid generated by the prefixes of the relators (similarly to unital, we are apply the label conservative to a factorisation in line with how it was used by Dolinka and Gray in \cite{DolGra21}).

	It may be noticed that we were particular about whether a certain factor appeared in its positive or negative form in our definition but that we use the phrase a conservative factorisation into $u_1, u_2, \ldots, u_k$ without reference to whether their inverses are or aren't used.
The reason for that may be seen in the next lemma.

\begin{lemma}	\label{lem:prefixMonConFactorForm}
	Let $G = \gpPres[r_i(u_1, \ldots, u_k) = 1 \, (i \in I)]$ be a group and let the set $\Gamma \subseteq \lbrace (j,\varepsilon) \, | \, 1 \leq j \leq k, \varepsilon \in \lbrace -1, 1 \rbrace \, \rbrace$ be minimal such that 
$r_i \in U^\ast$, for every $i \in I$,
where $U = \lbrace u_j^\varepsilon \mid (j, \varepsilon) \in \Gamma \rbrace$.
Then
\begin{equation}	\label{eqn:facForms}
	\subMon{ \bigcup_{(i, \varepsilon) \in \Gamma} \pref(u_j^{\varepsilon})}
	= \subMon{ \bigcup_{1 \leq j \leq k} \bigcup_{\varepsilon \in \lbrace -1, 1 \rbrace} \pref(u_j^{\varepsilon})} \leq G
\end{equation}
Consequently, if the $r_i$ have a conservative factorisation into $u_1, u_2, \ldots, u_k$ then
\begin{equation*}
	P = \subMon{ \bigcup_{1 \leq j \leq k} \bigcup_{\varepsilon \in \lbrace -1, 1 \rbrace} \pref(u_j^{\varepsilon})} \leq G
\end{equation*}
is equal to the prefix monoid of $G$.

\end{lemma}

\begin{proof}
	Let $Q$ be equal to the submonoid on the left hand side of equation (\ref{eqn:facForms}) and suppose there is some $j$ such that $u_j^{-1} \notin U$.
As each $u_j$ must appear in either its positive or its negative form at least once, there must be some $r_i$ such that $r_i \equiv v_p u_j v_s$ where $v_p,v_s \in U^\ast$.
In $G$, $r_i = 1$ which implies that $u_j^{-1} = v_s v_p$.
Thus $u_j^{-1} \in Q$ in $G$ and so $\pref(u_j^{-1}) = u_j^{-1} \cdot \pref(u_j) \in Q$.
Dually in any case where $u_j$ does not appear in the factorisation we may show that $\pref(u_j) \in Q$.

	This shows that the right hand side of equation (\ref{eqn:facForms}) is a subset of the left.
The reverse is obvious and so we have shown equality.
The second part is then immediate from the definition of a conservative factorisation.

\end{proof}

	The next result establishes a connection between unital and conservative factorisations of special presentations, extrapolating the result of \cite{DolGra21} from the one-relator case to finite presentations.

\begin{theorem}	\label{thm:UFisCF}

	Let $M = \invPres[r_i(u_1, u_2, \ldots, u_k) = 1 \, (i \in I)]$ be an inverse monoid and $G = \gpPres[r_i(u_1, u_2, \ldots, u_k) = 1 \, (i \in I)]$ be its maximal group image.
If the factorisation of the $r_i$ is unital in $M$ then it is conservative in $G$.
Furthermore if $M$ is E-unitary then it is also true that the factorisation of the $r_i$ being conservative in $G$ implies it is a unital factorisation in $M$.

\end{theorem}

\begin{proof}

	Suppose that the factorisation given by the $u_j$ is unital in $M$.
Let $q \in \pref( u_j^{\varepsilon} )$ for some $u_j^{\varepsilon}$ that appears exactly in at least one of the $r_i$.
Such an $r_i$ will have a prefix $p$ of the form $p \equiv w \cdot q$ for some $w \in \overline{ \lbrace u_1, u_2, \ldots, u_k \rbrace}^\ast$.
As $w$ is a concatenation of words which are unital in $M$ it will also be a unit when considered as an element of $M$.
Moreover as $r_i = 1$ we can see $p$ is a right unit, thus $q$ is a right unit in $M$ and so by Lemma \ref{lem:sigmaRMisP} belongs to the prefix monoid.
Therefore for all $1 \leq j \leq k$ and $\varepsilon \in \lbrace -1, 1 \rbrace$ such that $u_j^\varepsilon$ appears exactly in the factorisation the prefixes of $u_j^{\varepsilon}$ belong to the prefix monoid.
By Lemma \ref{lem:prefixMonConFactorForm} this means that the factorisation into $u_j$ is conservative.

	Suppose that $M$ is E-unitary and that the factorisation given by the $u_j$ is conservative in $G$.
Let $r_i$ be an arbitrary relator and $r_i \equiv v_1 v_2 \ldots v_{\ell}$ where $v_j \in \overline{\lbrace u_1, u_2, \ldots, u_k \rbrace}$ and the factorisation of $r_i$ into $v_j$ is in line with the overall factorisation.
First we observe that $r_i = 1$ which means $v_1^{-1} = v_2 \ldots v_{\ell}$ in $G$.
Further as the factorisation is conservative each factor, including the $v_j$, belongs to the prefix monoid.
Thus we have that $v_1^{-1} = p_1 \ldots p_n$ in $G$ where each $p_j$ is a prefix to one of the $r_i$.
As $M$ is E-unitary this means that $v_1^{-1} \sim p_1 \ldots p_n$, which in turn means that $p_1 \ldots p_n v_1 \in E_M$.
This is a product of right invertible elements and so is itself right invertible, and there is only one right invertible idempotent, the identity element.
So $p_1 \ldots p_n v_1 = 1$ in $M$, which means $v_1$ is left invertible and so as we already know it to be right invertible we have shown $v_1$ is invertible.
Hence we deduce that $v_2 v_3 \ldots v_{\ell} v_1 = 1$ and from here we may repeat the process to show that $v_2, \ldots, v_\ell$ are also invertible.
As our choice of $r_i$ was arbitrary this means that we can show any factor word $u_j$ is invertible and therefore the factorisation is unital.

\end{proof}

\section{Factorisations with Unique Marker Letters}

	In this section we will expand a result of Dolinka and Gray \cite[Theorem $5.1$]{DolGra21} which showed that for a conservative factorisation each factor word having a unique marker letter is sufficient to solve the prefix membership problem in the one-relator case.
We will show that by simply imposing the condition that the group has decidable word problem this result also holds for all finitely presented groups.
Before showing the sections main result we will dissect what such a group looks like.

	We call a factorisation \textit{uniquely marked} if the factor words, $u_1, u_2, \ldots, u_k \in \freemon$, are such that each $u_i$ has a corresponding $x_i \in X$ which appears precisely once in $u_i$ and not at all in any $u_j$ where $j \neq i$.
For instance $axbaybb^{-1}x^{-1}a^{-1}$ may be factorised as $(axb)(ayb)(axb)^{-1}$ which is uniquely marked as $ayb$ is marked by $y$ and $axb$ is marked by $x$, note that both the positive and negative forms of $axb$ are allowed to appear.

	We introduce the following lemma describing how a uniquely marked factorisation effects the structure of a group.
It is implicit in the work of Dolinka and Gray \cite{DolGra21} and Gray and Ruskuc \cite{GraRus23}, but for the sake of completeness we include an explicit proof here.

\begin{lemma}	\label{lem:UMLstructure}

	Let $G = \gpPres[r_i(u_1, u_2, \ldots, u_k) \, (i \in I)]$ be a group where $u_1, u_2, \ldots, u_k \in \freemon$ are a set of uniquely marked words.
Then $G$ is isomorphic to the free product of $H = \gpPres[r_i(z_1, z_2, \ldots, z_k) \, (i \in I)][z_1, z_2, \ldots, z_k]$ and a free group.
Furthermore $H$ is isomorphic to $\subGp{u_1, u_2, \ldots, u_k} \leq G$ under the mapping that sends $z_i$ to $u_i$.

\end{lemma}

\begin{proof}

	We will begin with a series of Tietze transformations which will give us a new form for the group.
We equate each factor word, $u_j$, with a new generator, $z_j$, giving
\begin{equation*}
	\gpPres[r_i(u_1, u_2, \ldots, u_k) = 1 \, (i \in I), u_j = z_j \, (1 \leq j \leq k)][X, Z]
\end{equation*}
where $Z = \lbrace z_1, z_2, \ldots, z_k \rbrace$.

	As $u_1, u_2, \ldots, u_k$ are uniquely marked we may divide the alphabet $X$ into two parts, $Y$, the unique marker letters and $X^\prime = X \setminus Y$.
For each factor we may write $u_j \equiv p_j y_j q_j$ where $y_j \in \overline{Y}$ and $p_j, q_j \in \freemon[X^\prime]$.
Applying this we may easily derive the following
\begin{equation*}
	\gpPres[r_i(z_1, z_2, \ldots, z_k) = 1 \, (i \in I), y_j = p_j^{-1} z_j q_j^{-1} \, (1 \leq j \leq k)][X, Z]
\end{equation*}
from the above.

	From this it is clear that the $y_j$ generators are now redundant, because as unique marker letters they will not have appeared in $p_j$ and $q_j$ which are written over $X^{\prime} = X \setminus Y$.
This gives the presentation
\begin{equation*}
	\gpPres[r_i(z_1, z_2, \ldots, z_k) = 1 \, (i \in I)][X^{\prime}, Z]
\end{equation*}
which has no relations on $x^{\prime}$.
Therefore we may write
\begin{equation*}
	\gpPres[r_i(z_1, z_2, \ldots, z_k) = 1 \, (i \in I)][Z] \ast \FG(X^{\prime}).
\end{equation*}

	This is isomorphic to $G$ under the mapping induced by sending $z_j$ to $u_j$ and $x_j$ to $x_j$ for $z_j \in Z$ and $x_j \in X^{\prime}$.
Hence $\gpPres[r_i(z_1, z_2, \ldots, z_k) = 1 \, (i \in I)][Z]$ is embedded within $G$.

\end{proof}

	We introduce a theorem of Dolinka and Gray \cite[Theorem A]{DolGra21} regarding the decidability of membership in certain submonoids of amalgamated groups products.

\begin{theorem}	\label{thm:subMonOfAmalGpProd}
	Let $G = B \ast_A C$ be a group, where $A$, $B$ and $C$ are finitely generated groups such that both $B$ and $C$ have decidable word problem and that membership in $A$ is decidable within both $B$ and $C$.
Let $M$ be a submonoid of $G$ such that the following hold:
\begin{enumerate}
	\item both $M \cap B$ and $M \cap C$ are finitely generated and
	\begin{equation*}
		M = \subMon{ (M \cap B) \cup (M \cap C) } \leq G;
	\end{equation*}
	\item membership in $M \cap B$ is decidable within $B$;
	\item membership in $M \cap C$ is decidable within $C$
	\item and that $A \subseteq M$.
\end{enumerate}
Then membership in $M$ within $G$ is decidable.
\end{theorem}

	We can now apply this and our previous structural result to find conditions under which the prefix membership problem is decidable.
\begin{theorem}	\label{thm:PMPforUML}

	Let $M = \invPres[r_i(u_1, u_2, \ldots, u_k) = 1 \, (i \in I)]$ be an inverse monoid and let $G = \gpPres[r_i(u_1, u_2, \ldots, u_k) = 1 \, (i \in I)]$ be its maximal group image.
If the words $u_1, u_2, \ldots, u_k \in \freemon$ are uniquely marked and form a conservative factorisation in $G$ (in particular if the factorisation is unital in $M$) and $G$ has decidable word problem then $G$ has decidable prefix membership problem and $M_{EU}$ has decidable word problem (in particular $M$ has decidable word problem if it is E-unitary).

\end{theorem}

\begin{proof}
	We carry over the notation from Lemma \ref{lem:UMLstructure}.

	By Lemma \ref{lem:prefixMonConFactorForm} we know that the prefix monoid takes the form
\begin{equation*}
	P = \subMon{ \bigcup_{1 \leq j \leq k} \bigcup_{\varepsilon \in \lbrace -1, 1 \rbrace} \pref(u_j^{\varepsilon})} \leq G
\end{equation*}

	We saw in Lemma \ref{lem:UMLstructure} that $G \cong H \ast \FG(X^{\prime})$.
We now consider the monoid $P^{\prime} \leq H \ast \FG(X^{\prime})$ which isomorphic to $P \leq G$ under the mapping sending $u_i$ to $z_i$.
We do this by looking at where the generators of $P$ are sent.
Recalling that $u_j \equiv p_j y_j q_j$, we may see that
\begin{equation*}
	\pref(u_j) = \pref(p_j) \cup p_j y_j \cdot \pref(q_j).
\end{equation*}
As our mapping from $G$ to $H \ast \FG(X^{\prime})$ sends $u_j$ to $z_j$ it also sends $y_j$ to $p_j^{-1} z_j p_j^{-1}$.
Using this we see that $\pref(u_j)$ is sent to
\begin{equation*}
	\pref(u_j) 
	= \pref(p_j) \cup z_j q_j^{-1} \cdot \pref(q_j)
	=\pref(p_j) \cup z_j \cdot \pref(q_j^{-1})
\end{equation*}
Dually, we may find that $\pref(u_j^{-1})$ is sent to $\pref(q_j^{-1}) \cup z_j^{-1} \cdot \pref(p_j)$.
Thus if we define the monoid
\begin{equation*}
	Q =
	\subMon{ \bigcup_{1 \leq j \leq k} \left( \pref(p_j) \cup \pref(q_j^{-1}) \right) }
	\leq \FG(X^\prime)
\end{equation*}
we may write that
\begin{equation*}
	P^{\prime} = \subMon{\overline{Z} \cup Q} \leq H \ast \FG(X^{\prime}).
\end{equation*}

	As $P^{\prime}$ is the submonoid of a free product we can apply Theorem \ref{thm:subMonOfAmalGpProd}.
Further, as this free product has no amalgamation, the conditions on $A$ in the Theorem are trivially satisfied.
It is clear that $P^{\prime} \cap H = \subMon{\overline{Z}} = H$ and $P \cap \FG(X^{\prime}) = \subMon{Q}$, and that both of these are finitely generated.
It follows therefore that $P^{\prime} = \subMon{(P^{\prime} \cap H) \cup (P \cap \FG(X^{\prime})}$, satisfying the first condition.
As membership in $H$ as a submonoid of itself is trivially decidable and as all submonoids have decidable membership in the free group $\FG(X^{\prime})$ by Benois's Theorem \cite{Ben69} the second and third conditions are also satisfied.
Thus $P^\prime$ has decidable membership in $H \ast \FG(X^{\prime})$ which by isomorphism means that $P$ has decidable membership in $G$.
Therefore $G$ has decidable prefix membership problem and so by Theorem \ref{thm:PMP+GWP=IWP} the word problem of $M_{EU}$ is decidable.

\end{proof}

	The next lemma, which is due to Magnus, may be found in Lyndon and Schupp \cite[Theorem IV.5.3]{LynSch01}, though it's given form here has been restricted to the case of finitely generated groups.

\begin{lemma}	\label{lem:CyRed-LettersGenDecSubGps}
	Let $G = \gpPres[w=1]$ be a group where $w$ is a cyclically reduced word and let $X^\prime \subseteq X$.
Then membership in $\subGp{X^\prime}$ within $G$ is decidable.
\end{lemma}

We now give an example of how Theorem \ref{thm:PMPforUML} may be applied to solve the word problem for a finitely presented special inverse monoid.

\begin{example}
	Consider the inverse monoid
\begin{equation*}
	M = \invPres[(z_1)(x_1^2y_1)^2(z_1)=1, (z_2)(x_2^2y_2)^2(z_2)=1, (z_1)(z_2)^{-1}=1]
\end{equation*}
 and its maximal group image
\begin{equation*}
	G = \gpPres[(z_1)(x_1^2y_1)^2(z_1)=1, (z_2)(x_2^2y_2)^2(z_2)=1, (z_1)(z_2)^{-1}=1].
\end{equation*}
The group $G$ is the product of $G_i = \gpPres[(z_i)(x_i^2y_i)^2(z_i)=1][x_i,y_i,z_i]$ (for $i=1,2$) amalgamated over $z_1 = z_2$.
Thus to show $G$ has decidable word problem it suffices to show that the $G_i$ have decidable word problem and that membership in $\subGp{z_i}$ within $G_i$ is decidable.
The $G_i$ are one-relator groups and so have decidable word problem by Magnus's Theorem.
The membership questions are decidable by Lemma \ref{lem:CyRed-LettersGenDecSubGps}.
We have already seen by Example \ref{exm:EuForUMLexample} that $M$ is E-unitary and the brackets demarcate a uniquely marked unital factorisation.
Therefore we may apply Theorem \ref{thm:PMPforUML} and find that $G$ has decidable prefix membership problem and $M$ has decidable word problem.
\end{example}

\begin{remark}
	It is important to note here that given inverse monoids $M_1$ and $M_2$ with maximal group images $G_1$ and $G_2$ that two words $u_1$ and $u_2$ generating isomorphic subgroups in $G_1$ and $G_2$ respectively does not in general imply that they will generate isomorphic inverse submonoids in $M_1$ and $M_2$ respectively.
This means that there may be cases where the group amalgam $G_1 \ast_{u_1 = u_2} G_2$ makes sense as a construct but the inverse monoid amalgam $M_1 \ast_{u_1 = u_2} M_2$ does not.
In our case however the $u_1$ and $u_2$ we are amalgamating over invertible pieces in their respective special inverse monoids and therefore $\subInv{u_i} \leq M_i$ is equal to $\subGp{u_i} \leq G_i$ and so the amalgamation of the two inverse monoids is valid.
\end{remark}

	When the factorisation is minimal and the inverse monoid is E-unitary inverse we can strengthen the results of Theorem \ref{thm:PMPforUML} but to do this we need to introduce the following lemma.

\begin{lemma}	\label{lem:MIPgenUM}
	Let $M = \invPres[r_i = 1 \, (i \in I)]$ be an inverse monoid and let $G = \gpPres[r_i = 1 \, (i \in I)]$ be its maximal group image.
Then the minimal invertible pieces $u_1, \ldots, u_k$ generate the group of units $U_M$ 
\end{lemma}

	The above is found in Dolinka and Gray \cite[Proposition $3.1$]{DolGra21}, and is a slight generalisation of arguments found in Ivanov, Margolis and Meakin \cite[Proposition $4.2$]{IMM01}.

\begin{corollary}

	Let $M = \invPres[r_i(u_1, u_2, \ldots, u_k) \, (i \in I)]$ be an E-unitary inverse monoid, $G = \gpPres[r_i(u_1, u_2, \ldots, u_k) \, (i \in I)]$  be its maximal group image and $U_M$ be its group of units.
If $u_1, u_2, \ldots, u_k$ are the minimal invertible pieces of $M$ and are uniquely marked then the following are equivalent:

\begin{enumerate}
	\item The word problem of $G$ is decidable.
	\item The word problem of $M$ is decidable.
	\item The word problem of $U_M$ is decidable.
\end{enumerate}

\end{corollary}

\begin{proof}
	The minimal invertible pieces inherently form a unital factorisation and so $(1)$ implies $(2)$ by Theorem \ref{thm:PMPforUML}.
As $U_M$ is a submonoid of $M$, $(2)$ implies $(3)$.
We know by Lemma \ref{lem:UMLstructure} that $G$ is isomorphic to the free product of a free group and the subgroup generated by the $u_i$.
By Lemma \ref{lem:MIPgenUM} the $u_i$ being minimal invertible pieces implies that they generate the subgroup $U_M$.
Thus $G$ is the free product of $U_M$ and a free group.
By Lemma \ref{lem:freeProdDec} the free product of two groups with decidable word problem has decidable word problem.
Therefore $(3)$ implies $(1)$ and so we have equivalence.
\end{proof}

	We now introduce a theorem from a paper of Gray and Ruskuc \cite[Theorem $3.2$]{GraRus23}.

\begin{theorem}		\label{thm:NikBobRewrite}
	Let $M = \invPres[r_i(u_1, u_2, \ldots, u_k) = 1 \, (i \in I)]$ be an inverse monoid, where the $u_i$ are all reduced words and form a unital factorisation.
Further, let
\begin{equation*}
\phi: \FG(Y) \rightarrow \subGp{u_1, u_2, \ldots, u_k} \leq \FG(X)
\end{equation*} 
be an epimorphism and let $u_1^\prime, u_2^\prime \ldots, u_k^\prime \in \FG(Y)$ be such that $\phi(u_i^\prime) = u_i$ in $\FG(X)$, for $1 \leq i \leq k$.
Finally, let $\psi: \freemon[Y] \rightarrow \freemon$ be the homomorphism produced by extending $y \mapsto \phi(y) \, (y \in \overline{Y})$.
Then
\begin{equation*}
	M = \invPres[r_i(\psi(u_1^\prime), \psi(u_2^\prime), \ldots, \psi(u_k^\prime)) = 1 \, (i \in I)]
\end{equation*}
and $\lbrace \psi(y) \mid y \in \overline{Y} \rbrace$ is a set of invertible pieces of $M$ which generates the same subgroup of $M$ as $u_1, \ldots, u_k$.
\end{theorem}

	The statement of this theorem is technical.
The key point is that there is sometimes a way to find an alternative presentation for an inverse monoid which decomposes into a more convenient set of unital factors.
For convenience we will render the particular way we intend to use it as a corollary.

\begin{corollary}	\label{cor:changeUnitsOfFac}
	Let $M = \invPres[r_i(u_1, \ldots, u_k) = 1 \, (i \in I)]$ be an inverse monoid and let the $r_i$ have a unital factorisation into a set of reduced words $U = \lbrace u_1, \ldots, u_k \rbrace \subset \freemon$.
Let $V = \lbrace v_1, \ldots, v_l \rbrace \subset \freemon$ be a set of words such that $\subGp{V} = \subGp{U} \leq \FG(X)$.
Then there are $r_i^\prime \in \freemon$ such that
\begin{equation*}
	M = \invPres[r_i^\prime = 1 \, (i \in I)],
\end{equation*}
that each $r_i^\prime$ reduces to $r_i$ and that the $r_i^\prime$ have a unital factorisation into $V$.
\end{corollary}

\begin{proof}

	Let the set of words $w_j(v_1, \ldots, v_l) \in \freemon$ be such that $w_j(v_1, \ldots, v_l) = u_j$ in $\FG(X)$, for $1 \leq j \leq k$. 
Let $Y = \lbrace y_1, \ldots, y_l \rbrace$ and let $u_j^\prime \equiv w_j(y_1, \ldots, y_l)$.
Finally, let
\begin{equation*}
	\phi: \FG(Y) \rightarrow \subGp{U} \leq \FG(X)
	\text{ and }
	\psi: \freemon[Y] \rightarrow \freemon
\end{equation*}
be the maps sending $y_i$ to $v_i$ as elements and words respectively.

	We can see that $\phi(\FG(Y)) = \subGp{V} \leq \FG(X)$ by our definition of $\phi$.
By assumption $\subGp{V} = \subGp{U} \leq \FG(X)$ and so $\phi$ is an epimorphism.
Moreover 
\begin{equation*}
	\phi(u_j^\prime) = \phi(w_j(y_1, \ldots, y_l)) = w_j(v_1, \ldots, v_l) = u_j
\end{equation*}
in $\FG(X)$ and $\overline{V} = \lbrace \psi(y) \mid y \in \overline{Y} \rbrace$.
Therefore we satisfy all the conditions of Theorem \ref{thm:NikBobRewrite}.

	Hence if we let 
\begin{equation*}	
	r_i^\prime 
	= r_i(\psi(u_1^\prime), \ldots, \psi(u_k^\prime)) 
	= r_i( w_1(v_1, \ldots, v_k), \ldots, w_k(v_1, \ldots,v_k) )
\end{equation*}
we have that $M = \invPres[r_i^\prime = 1 \, (i \in I)]$.
Moreover we have that the set $\overline{V}$ is composed of invertible pieces in $M$ and therefore the $r_i^\prime$ have a unital factorisation into the $V$ in $M$.
We also observe that as each $u_j$ is reduced it must be the case that if $w_j(v_1, \ldots, v_l) = u_j$ then $\red(w_j(v_1, \ldots, v_l)) \equiv u_j$, thus $r_i^\prime$ reduces to $r_i$.
\end{proof}

\begin{remark}
	Though each $r_i^{\prime}$ must reduce to its respective $r_i$, we do not strictly require that the $r_i$ are reduced words.
\end{remark}

	In particular, this allows us to identify a family of special inverse monoids whose initial presentations do not possess a uniquely marked unital factorisation but which may be differently presented so that there is a unital and uniquely marked factorisation.
This generalises a result of Dolinka and Gray \cite[Theorem $5.4$]{DolGra21}

\begin{theorem}	\label{thm:hiddenUML}
	Let $M = \invPres[r_i(u_1, u_2, \ldots, u_k) = 1 \, (i \in I)]$ be an inverse monoid and the factorisation into $u_j$ be unital.
Further let the maximal group image of $M$ be $G =\gpPres[r_i(u_1, u_2, \ldots, u_k) = 1 \, (i \in I)]$ and suppose it has decidable word problem.
Finally, for $1 \leq j \leq \ell$, let there be $x_j, z_j \in X_j \subseteq X$, $Y_j = X_j \setminus \lbrace x_j, z_j \rbrace$ and $W_j \subset \freemon[Y_j]$ such that the following hold:
\begin{enumerate}
	\item The sets $X_{j_1}$ and $X_{j_2}$ are equal if and only if $j_1 = j_2$ and are disjoint otherwise.
	\item The sets $\lbrace u_1, u_2, \ldots, u_k \rbrace$ and $\bigcup_{j=1}^{\ell} \lbrace x_j \cdot w \cdot z_j \mid w \in W_j \rbrace$ are equal.
	\item The empty word belongs to each $W_j$, for $1 \leq j \leq \ell$.
	\item For each $y \in Y_j$ there is some $w_y \in \freemon[W_j]$ such that $ y \equiv \red(w_y)$, for $1 \leq j \leq \ell$.
\end{enumerate}

	Then $G$ has decidable prefix membership problem and $M_{EU}$ has decidable word problem (in particular $M$ has decidable word problem if it is E-unitary).

\end{theorem}

\begin{proof}

	Let $V = \bigcup_{j=1}^{\ell} \lbrace x_j y x_j^{-1} \mid y \in Y_j \rbrace \cup \lbrace x_j z_j \rbrace$ and $U = \lbrace u_1, u_2, \ldots, u_k \rbrace$.
We know that each $u_i$ is of the form $x_j \cdot w \cdot z_j$ for some $w \in W_j \subset \freemon[Y]$.
Therefore if $w \equiv y_1 y_2 \ldots y_n$ we can say that $u_i$ is equal in the free group $\FG(X_j)$ to $x_j y_1 x_j^{-1} \cdot x_j y_2 x_j^{-1} \cdot \ldots \cdot x_j y_n x_j^{-1} \cdot x_j z_j$.
Thus $\subGp{U} \subseteq \subGp{V} \leq \FG(X)$.

	Each $y \in Y_j$ has a corresponding $w_y \in \freemon[W_j]$ which reduces to $y$.
This means any $v \equiv x_j y x_j^{-1}$ is equal to $x_j w_y x_j^{-1}$ in the free group $\FG(X_j)$.
In the free group this $x_j w_y x_j^{-1}$ will be a product of $x_j w x_j^{-1}$ for $w \in W$.
Thus as $x_j w x_j^{-1} = x_j w z_j (x_j z_j)^{-1} \in \subGp{U}$, it follows that $x_j y x_j^{-1} \in \subGp{U}$.
The only $v \in V$ not of the form $x_j y x_j^{-1}$ is $x_j z_j$ which already amongst the $u_i$.
Thus $\subGp{V} \subseteq \subGp{U} \leq \FG(X)$ and therefore $\subGp{U} = \subGp{V} \leq \FG(X)$.

	This means that we can apply Corollary \ref{cor:changeUnitsOfFac}.
This means that there are some $r_i^\prime$ such that $M^\prime = \invPres[r_i^\prime = 1 \, (i \in I)]$ is equal to $M$ and further it means that the $r_i^\prime$ have a factorisation into $V$ which is unital in $M$ and uniquely marked.
By Theorem \ref{thm:PMPforUML} we can now decide membership for the prefix monoid of the group $G^\prime = \gpPres[r_i^\prime = 1 \, (i \in I)]$.

	Suppose that some word $w \in \freemon$ belongs to this monoid when considered as an element of $G^\prime$.
By Lemma \ref{lem:sigmaRMisP} we know that this means $w$ is right invertible as an element of $M^\prime$.
As $M^\prime = M$ this means that $w$ is right invertible as an element of $M$ too, therefore by a second application of Lemma \ref{lem:sigmaRMisP} $w$ belongs to the prefix monoid of $G$.
Thus we can decide the prefix membership problem for $G$ and therefore by Theorem \ref{thm:PMP+GWP=IWP} the word problem for $M_{EU}$.

\end{proof}

We give will now give an example of applying the above result.

\begin{example}
	For $i = 1, 2$, let
\begin{equation*}
	M_i = \invPres[ (a_i b_i c_i d_i)(a_i c_i d_i)(a_i d_i)(a_i b_i b_i c_i d_i)(a_i c_i d_i) = 1][a_i, b_i, c_i, d_i].
\end{equation*}
These are copies of the \textit{O'Hare monoid} (see \cite{MMS87} for more, including an implicit proof of why the above factorisation is unital).
These monoids are both E-unitary by Theorem \ref{thm:cycRedIsEu} as they have one cyclically reduced relator.
Let $G_1$ and $G_2$ be the maximal group images of $M_1$ and $M_2$ respectively, they will be one relator and thus have decidable word problem by Magnus's Theorem.
Further let $M = M_1 \ast_{a_1 b_1 b_1 c_1 d_1 = a_2 d_2} M_2$, by Theorem \ref{thm:EuAmalByUM} this is E-unitary and has a special finite presentation 
\begin{align*}
	\Inv \langle a_i, b_i, c_i, d_i \mid
	(a_i b_i c_i d_i)(a_i c_i d_i)(a_i d_i)(a_i b_i b_i c_i d_i&)(a_i c_i d_i) = 1 
	\text{ for } i = 1,2, \\
	&(a_1 b_1 b_1 c_1 d_1) (a_2 d_2)^{-1} = 1
	\rangle
\end{align*}
	The maximal group image of $M$ has decidable word problem by Lemma \ref{lem:amalProdDec} as it is the amalgamation of two groups with decidable word problem, $G_1$ and $G_2$, over finitely generated subgroups with decidable membership within their respective groups, $\subGp{a_1 b_1 b_1 c_1 d_1} \leq G_1$ and $\subGp{a_2 d_2} \leq G_2$
(these are decidable by Benois' Theorem as the $G_i$ are free, see \cite[Remark $5.5$]{DolGra21}).

	If we let $X_i = \lbrace a_i, b_i, c_i, d_i \rbrace$, $Y_i = \lbrace b_i, c_i \rbrace$ and $W_i = \lbrace b_i c_i, c_i, 1, b_i b_i c_i \rbrace$ then it is easily seen that $M$ satisfies the conditions for Theorem \ref{thm:hiddenUML} and therefore has decidable word problem.
\end{example}

\section{Factorisations with Disjoint Alphabets}

	In this section we will proceed largely along parallel lines to the previous one though focusing on a different property that the factorisation may possess.

	We say that a factorisation $w_i(u_1, u_2, \ldots, u_k) \, (i \in I)$ of words in $\freemon$ is \textit{alphabetically disjoint} if $k \geq 2$ and for all $1 \leq j_1, j_2 \leq k$ such that $j_1 \neq j_2$ the factors $u_{j_1}$ and $u_{j_2}$ have no letters in common (with $x$ and $x^{-1}$ counting as a common letter).
We require that $k \geq 2$ as otherwise the condition becomes too broad.

	We now consider the implications of a group possessing such a factorisation.
Similarly to the first lemma of the previous section, this lemma compiles some results implicit in \cite{DolGra21} and \cite{GraRus23} and provides a proof for the sake of completeness.

\begin{lemma}	\label{lem:DAstructure}
	
	Let $G = \gpPres[r_i(u_1, u_2, \ldots, u_k) = 1 \, (i \in I)]$ be a group where the factorisation into $u_1, u_2, \ldots, u_k \in \freemon[X]$ is alphabetically disjoint.
Further let $X_1, \ldots, X_k \subset X$ be sets where each $X_j$ is the minimal alphabet such that $u_j \in \freemon[X_j]$ and let $X_0 = X \setminus \bigcup_{j=1}^{k} X_j$.

	Let
\begin{equation*}
	G_j = \gpPres[r_i( z_1, \ldots z_j, u_{j+1}, \ldots, u_k) = 1 \, (i \in I)][X_0, z_1, \ldots, z_j, X_{j+1}, \ldots, X_k]
\end{equation*}
for $0 \leq j \leq k$.
Then the order of $u_j$ is the same in all $G_i$ where $0 \leq i < j$ and also equal to the order of $z_j$ in all $G_i$ where $j \leq i \leq k$.

	For $0 \leq j \leq k$, let $m_j$ be the order of $u_j$ in $G_{j-1}$ if this is finite and $0$ otherwise.
Further let $B_j = \gpPres[u_j^{m_j} = 1][X_j]$, and thus $B_j = \FG(X_j)$ if $u_j$ is of infinite order, be a second series of groups.

	These are such that $G = G_0$, $G_k = H \ast \FG(X_0)$ where
\begin{equation*}	
	H = \gpPres[r_i(z_1, \ldots, z_k) = 1 \, (i \in I)][z_1, \ldots, z_k], 
\end{equation*}
and that
\begin{equation*}
	G_{j-1} = G_j \ast_{A_j} B_j
\end{equation*}
where $A_j$ represents amalgamation by $u_j = z_j$ over the isomorphic subgroups $\subGp{u_j} \leq B_j$ and $\subGp{z_j} \leq G_j$.
Consequently $H$ embeds naturally into $G$ under the mapping $z_i$ to $u_i$.

\end{lemma}

\begin{proof}

	The claims that $G_0 = G$ and $G_k = H \ast \FG(X_0)$ are immediate consequences of the definitions.

	We know that $G_{j-1}$ takes the form
\begin{equation*}
	\gpPres[r_i(z_1, \ldots, z_{j-1}, u_j, \ldots u_k) = 1 \, (i \in I)][X_0, z_1, \ldots z_{j-1}, X_j, \ldots, X_k].
\end{equation*}
We now introduce a new generator $z_j$ and set it equal to the word $u_j$.
We can also introduce the relation $u_j^{m_j} = 1$ without changing the group as the order of $u_j$ is already $m_j$ by assumption.
Together these changes give
\begin{align*}
	\Gp \langle X_0, z_1, & \ldots z_{j-1}, X_j, \ldots, X_k \mid \\
	&r_i(z_1, \ldots, z_{j-1}, u_j, \ldots u_k) \, (i \in I),
	z_j = u_j, u_j^{m_j} = 1 \rangle
\end{align*}
which may be separated out into the amalgamated product
\begin{align*}
	\gpPres[u_j^{m_j} = 1][X_j]
	\ast_{u_j = z_j} 
	\Gp \langle X_0, z_1, & \ldots z_{j}, X_{j+1}, \ldots, X_k \mid \\
	&r_i(z_1, \ldots, z_{j}, u_{j+1}, \ldots u_k) \, (i \in I) \rangle
\end{align*}
The first part of this is equal $G_j$, the second part to $B_j$ and the amalgamation is over the subgroups generated by $z_j$ and $u_j$ respectively, matching $A_j$.
By \cite[Theorem IV.$5.2$]{LynSch01} if a group of the form $\gpPres[u_j^{m_j} = 1][X_j]$ then $u_j$ has order precisely $m_j$ if $m_j$ finite.
The order of $z_j$ in $G_j$ will be $m_j$ as a consequence of the relations $r_i( z_1, \ldots z_j, u_{j+1}, \ldots, u_k) = 1$ just as $u_j$ was of order $m_j$ in $G_{j-1}$ as a consequence of the relations $r_i( z_1, \ldots z_{j-1}, u_j, \ldots, u_k) = 1$.
Therefore $z_j$ and $u_j$ generate isomorphic subgroups in $G_j$ and $B_j$ respectively and so the amalgamation is valid.

	Finally, we note that as $G_j$ embeds in $G_{j-1}$, for $1 \leq j \leq k$, under the embedding $x \mapsto x$, for $x \in X \setminus \left( \Sigma_{i=1}^j X_i \right)$, $z_i \mapsto z_i$, for $1 \leq i \leq j - 1$ and $z_j \mapsto  u_j$ it follows that the $u_j$ (and $z_j$) have consistent order across the $G_i$ in the sense described.
In particular this order is either $m_j$ if $m_j > 0$ or is not of finite order otherwise. 

\end{proof}

	Before we can approach the prefix membership problem we first need to introduce another result \cite[Lemma $5.6$]{DolGra21}.

\begin{lemma}	\label{lem:genOfSubMonOfAmal}
	Let $G = B \ast_A C$ and let $U$ be a finite subset of $B \cup C$ such that $M = \subMon{U}$ contains $A$.
Then $M \cap B$ and $M \cap C$ are generated by $(U \cap B) \cup A$ and $(U \cap C) \cup A$ respectively.
\end{lemma}

	Similarly to the previous section, we now apply our understanding of the structure to provide a set of conditions on the prefix monoid of a group with such a factorisation.

\begin{theorem}	\label{thm:PMPforDA}

	Let $M = \invPres[r_i(u_1, u_2, \ldots, u_k) = 1 \, (i \in I)]$, where $k > 1$, be an inverse monoid and let $G = \gpPres[r_i(u_1, u_2, \ldots, u_k) = 1 \, (i \in I)]$ be its maximal group image.

	If all the following hold
\begin{itemize}
	\item The factorisation into $u_1, u_2, \ldots, u_k \in \freemon$ is alphabetically disjoint and conservative in $G$ (in particular if they form a unital factorisation in $M$);
	\item The word problem for $G$ is decidable;
	\item For every factor $u_j$ of finite order, $m_j$, in $G$ the group $\gpPres[u_j^{m_j} = 1]$ has decidable prefix membership problem;
	\item For every factor $u_j$ not of finite order in $G$ the subgroup generated by $z_j$ has decidable membership within $G_j$ (where $z_j$ and $G_j$ are as defined in Lemma \ref{lem:DAstructure})
\end{itemize}
then $G$ has decidable prefix membership problem.
Furthermore $M_{EU}$ has decidable word problem (in particular $M$ has decidable word problem if it is E-unitary).

\end{theorem}

\begin{proof}
	We carry over notation from Lemma \ref{lem:DAstructure}.

	Let $U_j = \lbrace z_{j_1}^{\varepsilon}, \pref(u_{j_2}^{\varepsilon}) \mid 1 \leq j_1 \leq j < j_2 \leq k, \varepsilon \in \lbrace -1, 1 \rbrace \rbrace$ for $0 \leq j \leq k$, also let $M_j = \subMon{U_j} \leq G_j$.
We claim the following:
\begin{enumerate}
	\item That $\subMon{U_{j-1}} = \subMon{ U_j \cup \pref(u_j) \cup \pref(u_j^{-1}) } \leq  G_j \ast_{A_j} B_j$
	\item That $M_{j-1} \cap G_j = M_j \leq G_j \ast_{A_j} B_j= G_{j-1}$
	\item That $M_{j-1} \cap B_j = \pref(u_j) \cup \pref(u_j^{-1}) \leq G_j \ast_{A_j} B_j$
\end{enumerate}
for $1 \leq j \leq k$.

	That $U_{j-1} \subseteq U_j \cup \pref(u_j) \cup \pref(u_j^{-1})$ follows immediately from the definition of $U_j$.
As $z_j^{\varepsilon} = u_j^{\varepsilon} \in \pref(u_j^{\varepsilon})$ in $G_j \ast_{A_j} B_j$, we have that
\begin{equation*}
	U_j = U_j \cup \lbrace z_j, z_j^{-1} \rbrace \supseteq U_j \cup \pref(u_j) \cup \pref(u_j^{-1}).
\end{equation*}
Therefore $U_{j-1} = U_j \cup \pref(u_j) \cup \pref(u_j^{-1})$ in $G_{j-1}$ and $(1)$ is an immediate consequence.

	By $(1)$, $M_{j-1} = \subMon{U_j \cup \pref(u_j) \cup \pref(u_j^{-1})}$ in $G_j \ast_{A_j} B_j$.
Further by Lemma \ref{lem:genOfSubMonOfAmal} we have that $M_{j-1} \cap G_j$ is generated by
\begin{equation*}
	\left( \left( U_j \cup \pref(u_j) \cup \pref(u_j^{-1}) \right) \cap G_j \right) \cup A_j.
\end{equation*}
The generators $p \in \pref(u_j^{\varepsilon})$ where $p \neq u_j^{\varepsilon} = z_j^\varepsilon$ do not belong to $G_j$ and $A_j$ is generated by $z_j \in U_j$.
Hence $M_{j-1} \cap G_j$ is the submonoid of $G_j$ generated by $U_j$ and is therefore equal to $M_j$, which proves $(2)$.

	Similarly we may apply Lemma \ref{lem:genOfSubMonOfAmal} to say that $M_{j-1} \cap B_j$ is generated by
\begin{equation*}
	\left( \left( U_j \cup \pref(u_j) \cup \pref(u_j^{-1}) \right) \cap B_j \right) \cup A_j.
\end{equation*}
The only members of $U_j \cap B_j$ are $z_j = u_j$ and $z_j^{-1} = u_j^{-1}$ and these are also the generators of $A_j$.
Therefore $M_{j-1} \cap B_j$ is the submonoid generated by $\pref(u_j)$ and $\pref(u_j^{-1})$, which proves $(3)$.

	We shall now apply these claims to show inductively that each $M_j$ has decidable membership within its respective $G_j$.
Firstly, the monoid $M_k$ is equal to $H$, the group generated by all the $z_j$, and as $G_k = H \ast \FG(X_0)$ membership is decidable.

	Suppose that $M_j \leq G_j$ is decidable.
We will now show this implies that $M_{j-1} \leq G_{j-1}$ has decidable membership via use of Theorem \ref{thm:subMonOfAmalGpProd}.

	Firstly, we see that
\begin{equation*}
	M_{j-1} 
	= \subMon{ (M_{j-1} \cap G_j) \cup (M_{j-1} \cap B_j) }
	\leq G_j \ast_{A_j} B_j
	= G_{j-1}
\end{equation*}
as a consequence of the claims we proved above.

	The group $G_j$ has decidable word problem as it embeds naturally in $G$, which has decidable word problem by assumption.
The group $B_j$ is either a free group or a one relator group both classes which always have decidable word problem.
Both are finitely generated.

	The membership of $A_j$ in $B_j$ is decidable as either $B_j$ is a free group, in which case we may apply Benois\textsc{\char13}s Theorem, or $u_j$ has a finite order and as such $A_j$, which is generated $u_j$ alone, is a finite subgroup of $B_j$ and so $B_j$ having a decidable word problem suffices to decide membership.
The membership of $A_j$ in $G_j$ is decidable if the order of the $u_j$ is finite by the same argument as for $B_j$.
If $u_j$ is not of finite order then as $A_j$ is the subgroup of $G_j$ generated by $z_j$ which has decidable membership by assumption.

	The monoid $M_{j-1} \cap G_j$ is equal to the monoid $M_j$ which has decidable membership in $G_j$ by inductive hypothesis.
The monoid $M_{j-1} \cap B_j$ is equal to the monoid $\subMon{\pref(u_j) \cup \pref(u_j^{-1})}$.
If $u_j$ is of finite order in $G$ then $B_j = \gpPres[u_j^{m_j} = 1][X_j]$ and $(u_j)(u_j)\ldots(u_j)$ is a conservative factorisation of the relator.
Therefore by Lemma \ref{lem:prefixMonConFactorForm} the monoid $M_{j-1} \cap B_j$ is the prefix monoid of $B_j$ and which has decidable membership by assumption.
If $u_j$ is not of finite order then $B_j$ is a free group and all submonoids have decidable membership as a consequence of Benois' theorem.
Thus we have satisfied the conditions of Theorem \ref{thm:subMonOfAmalGpProd} and $M_{j-1}$ has decidable membership within $G_{j-1}$.

	Therefore by induction $M_0$ has decidable membership in $G_0 = G$ and as $M_0$ is generated by the prefixes of both the factors and the inverses of the factors of a conservative factorisation of $G$ it is equal to the prefix monoid by Lemma \ref{lem:prefixMonConFactorForm}.
Thus $G$ has decidable prefix membership problem.
\end{proof}

\begin{example}
	If we let $Y_i = \lbrace a_i, b_i, c_i \rbrace$, $u_i \equiv a_i^2 b_i^3$ and $v_i \equiv c_i^5$ for $i = 1, 2$.
Then we may define the inverse monoid
\begin{equation*}
	M = \invPres[u_i v_i^2 u_i v_i u_i = 1 \, (i=1,2), u_1 u_2^{-1} = 1][Y_1, Y_2]
\end{equation*}
(this is the same presentation as Example \ref{exm:EuForDAexample})
with maximal group image
\begin{equation*}
	G = \gpPres[u_i v_i^2 u_i v_i u_i = 1 \, (i=1,2), u_1 u_2^{-1} = 1][Y_1, Y_2].
\end{equation*}
It may easily be seen that the factorisation into $u_i$ and $v_i$ is alphabetically disjoint; that it is also unital in $M$ may be shown by applying the Adjan algorithm (see \cite[Section $3$]{DolGra21} for a fuller description but in this case it suffices to observe that the relators are units and then make repeated use of the rule that $w_1 w_2, w_2 w_3 \in U_M$ implies $w_1, w_2, w_3 \in U_M$).

	If we let $K_i = \gpPres[u_i v_i^2 u_i v_i u_i = 1][Y_i]$ (for $i=1,2$) then as the factorisation into $u_i$ and $v_i$ is alphabetically disjoint the subgroups $\subGp{u_i}$ and $\subGp{v_i}$ have decidable membership within $K_i$ by Lemma \ref{lem:CyRed-LettersGenDecSubGps} and the groups $K_i$ themselves have decidable word problem by Magnus's Theorem.
Therefore $G = K_1 \ast_{u_1=u_2} K_2$ has decidable word problem by Lemma \ref{lem:amalProdDec}.

	The words $u_i$ and $v_i$ are not of finite order within $K_i$ by the Freiheitssatz.
It follows from $K_i$ embedding in $G$ that the words are not of finite order in $G$ either.
Therefore there are no factors of finite order.

	The amalgamated parts of $K_1$ and $K_2$ are the groups generated by $u_1$ and $u_2$ respectively.
We have seen above that membership in these groups is decidable within their respective $K_i$.
By Theorem \ref{thm:subMonOfAmalGpProd} (or by use of the normal form theorem for amalgamated free products, see for instance \cite[Theorem IV.$2.6$]{LynSch01}) this is sufficient to decide membership in $K_i$ within $G$.
From there we can then decide membership in $\subGp{u_i}$ or $\subGp{v_i}$ within $K_i$ and thus within $G$.
By Lemma \ref{lem:DAstructure} the above is sufficient to decide membership in the relevant $\subGp{z_j}$ within $G_j$, as this problem will embed in the larger one.

Thus we may apply Theorem \ref{thm:PMPforDA} and find that $G$ has decidable prefix membership problem and, as we have already seen in Example \ref{exm:EuForDAexample} that $M$ is E-unitary, that $M$ has decidable word problem.

\end{example}

\begin{remark}
	We note that the original theorem \cite[Theorem $5.7$]{DolGra21} may be recovered from this result.
Suppose we have some monoid $M = \invPres[w=1]$ for a cyclically reduced $w$ which has an alphabetically disjoint factorisation.
As each factor is written over only a portion of the alphabet that the whole relator is written over by the Freiheitsatz we know that no factor can have finite order.
The condition on factors of non-finite order is then satisfied Lemma \ref{lem:CyRed-LettersGenDecSubGps} as $w$ is cyclically reduced.
Then we know that its maximal group image has decidable word problem by Magnus's Theorem and that it is E-unitary by Theorem \ref{thm:cycRedIsEu} as $w$ is cyclically reduced.
Therefore by Theorem \ref{thm:PMPforDA} $M$ has decidable word problem.
\end{remark}

	We now consider the relationship between the word problems of the $G_j$ from Lemma \ref{lem:DAstructure}.

\begin{lemma}		\label{lem:GnWPImpliesGmWP}

	Let $G = \gpPres[r_i(u_1, u_2, \ldots, u_k) = 1 \, (i \in I)]$ be a group with alphabetically disjoint factorisation.
For $1 \leq j \leq k$, let
\begin{equation*}
	G_j = \gpPres[r_i( z_1, \ldots z_j, u_{j+1}, \ldots, u_k) = 1 \, (i \in I)][X_0, z_1, \ldots, z_j, X_{j+1}, \ldots, X_k]
\end{equation*}
and let $1 \leq n,m \leq k$ and suppose that $G_n$ has decidable word problem.

	Then the word problem of $G_m$ is decidable if:
\begin{itemize}
	\item Either $n \leq m$
	\item Or $m < n$ and all $u_j$, where $m \leq j < n$, of not of finite order are such that $\subGp{z_j} \leq G_j$ has decidable membership.
\end{itemize}

	Consequently, if $\subGp{u_j} \leq G$ has decidable membership for all factors, $u_j$, of non-finite order then $H = \gpPres[r_i(z_1, \ldots, z_k) = 1 \, (i \in I)][z_1, \ldots, z_k]$ having decidable word problem implies that $G$ has decidable word problem.

\end{lemma}

\begin{proof}
	Suppose that $n \leq m$.
Then $G_m$ embeds naturally into $G_n$ and so deciding the latter's word problem is sufficient to decide the former.

	Suppose that $m < n$.
It suffices to demonstrate the case for $m = n - 1$, as the full result follows by induction.
Recall that $G_{n - 1} = G_{n} \ast_{A_n} B_n$.
By Lemma \ref{lem:amalProdDec} to show that this group amalgam has decidable word problem it  will suffice to demonstrate that the two groups being amalgamated have decidable problem and that membership in the amalgamating subgroup is decidable within each of these two groups.

	As $B_n$ is either free or a one relator group it has decidable word problem.
If $u_n$ has finite order in $B_n$, this means that $A_n$ which is generated by $u_n$ has finitely many members.
Otherwise $B_n$ is free and membership in $A_n$ is decidable by Benois' Theorem.

	By assumption $G_n$ has decidable word problem.
If the order of $z_n$, which (by Lemma \ref{lem:DAstructure}) is equal to the order of $u_n$, is finite then as above $A_n$ is finite and thus has decidable membership within $G_n$.
If however the order of $z_n$ is not finite the decidability of membership in $\subGp{z_n}$ within $G_n$ is decidable by assumption.

	Therefore $G_{n-1}$ has decidable word problem and inductively so does $G_m$.

	Hence if $H$ has decidable word problem, which is equivalent to the word problem of $G_k = H \ast \FG(X_0)$, and each of the non finite order factors generate a subgroup with decidable membership within $G$ then $G$ has decidable word problem.

\end{proof}

	We can now once again consider the word problem for the group of units of a suitably factorised inverse monoid.

\begin{corollary}	\label{cor:DAunits}

	Let $M = \invPres[r_i(u_1, u_2, \ldots, u_k) \, (i \in I)]$ be an E-unitary inverse monoid, let $G = \gpPres[r_i(u_1, u_2, \ldots, u_k) \, (i \in I)]$  be its maximal group image and let $U_M$ be its group of units.
If the factor words $u_1, u_2, \ldots, u_k$ are the minimal invertible pieces of $M$ and alphabetically disjoint and:
\begin{itemize}
	\item For every factor $u_j$ of finite order, $m_j$, in $G$ the group $\gpPres[u_j^{m_i} = 1]$ has decidable prefix membership problem;
	\item For every factor $u_j$ which does not have finite order in $G$ the membership of $\subGp{u_j} \leq G$ is decidable;
\end{itemize}
then the following are equivalent:

\begin{enumerate}
	\item The word problem of $G$ is decidable.
	\item The word problem of $M$ is decidable.
	\item The word problem of $U_M$ is decidable.
\end{enumerate}

\end{corollary}

\begin{proof}

	All the conditions of Theorem \ref{thm:PMPforDA} are satisfied, so $(1)$ implies $(2)$.
As $U_M$ is a submonoid of $M$, $(2)$ implies $(3)$.
By Lemma \ref{lem:MIPgenUM} the $u_i$ being minimal invertible pieces means that they generate $U_M$.
Moreover as they are invertible the structure of $U_M$ will be identical to the subgroup generated by the $u_i$ within $G$.
By Lemma \ref{lem:DAstructure} this is isomorphic to
\begin{equation*}
	H = \gpPres[r_i(z_1, \ldots, z_k) = 1 \, (i \in I)][z_1, \ldots, z_k]
\end{equation*}
and so by Lemma \ref{lem:GnWPImpliesGmWP} we may say that $(3)$ implies $(1)$.

\end{proof}

	At this point it is worth noting that determining the order of a word is not necessarily an easy matter.
In fact the general question of whether $x = y^n$ for some $n$, called the \textit{power problem}, is strictly harder than the word problem even if we restrict to when $x = 1$, in which case it is called the \textit{order problem}, see \cite{MC70} and \cite{Col73}.
This motivates questioning the necessity of the conditions which rely on referencing the order of the factors.
In particular it would be useful to show that they always held, at least under the circumstances of Corollary \ref{cor:DAunits}.

	Suppose that we consider $u_j$ which is of finite order in $G$.
This means that $u_j^{m_j} = 1$ in $G$, therefore $u_j^{m_j}$ is an idempotent in $M$ but it is also a product of units and therefore a unit itself.
The only idempotent unit in any inverse monoid is identity and so $u_j^{m_j} = 1$ in $M$ as well.
This means that if $u_j$ is a minimal invertible piece of $M$ then $u_j$ is a minimal invertible piece of $\invPres[u_j^{m_j} = 1][X_j]$ also, and hence we may deduce that there are no other pieces in the minimal factorisation.
Thus it would remove a condition from Corollary \ref{cor:DAunits} if there was a positive answer to the following question:
\begin{question}
	Does the group $\gpPres[u^m = 1]$, where $m \geq 1$, have decidable prefix membership problem if the only minimal invertible piece of the corresponding inverse monoid $\invPres[u^m = 1]$ is $u$?
\end{question}

	Suppose we consider when $u_j$ is of infinite order in $G$.
In the case of a group with a single cyclically reduced relator we already know by Lemma \ref{lem:CyRed-LettersGenDecSubGps} that a subgroup generated some subset of the relators has decidable membership.
We might ask:
\begin{question}
	For a group $G = \gpPres[r_i = 1 \, (i \in I)]$ with decidable word problem and $x \in X$, under what conditions is membership in $\subGp{x}$ within $G$ decidable?
\end{question}

	More generally, by answering the above questions or otherwise:

\begin{question}
	Is the word problem of an E-unitary special inverse monoid equivalent to the word problems of its maximal group image and group of units if the minimal invertible pieces of its relators are alphabetically disjoint? 
\end{question}

	Finally, sets of words being uniquely marked or alphabetically disjoint are just specific examples of properties which are what Gray and Ruskuc \cite[Definition $3.6$]{GraRus23} called being \textit{free for substitution}, so we might question whether this more general property is sufficient:

\begin{question}
	Let $M = \invPres[r_i(u_1, \ldots, u_k) = 1 \, (i \in I)]$ be an E-unitary inverse monoid and let $G = \gpPres[r_i(u_1, \ldots, u_k) = 1 \, (i \in I)]$ be its maximal group image.
Further let $u_1, \ldots, u_k$ be the minimal invertible pieces and suppose that
\begin{equation*}
	\gpPres[r_i(z_1, \ldots, z_k) = 1 \, (i \in I)][z_1, \ldots, z_k]
	\cong
	\subGp{u_1, \ldots, u_k} \leq G.
\end{equation*}
Are the word problems of $G$, $M$ and $U_M$ equivalent?
\end{question}

\bibliography{fpf-citations}{}
\bibliographystyle{plain}

\end{document}